\documentclass[11pt]{amsart}
\usepackage{graphicx}
\usepackage[shortlabels]{enumitem}
\usepackage{amssymb,amsmath,commath}
\usepackage{tikz,tikz-cd,soul,eucal,mathrsfs,bbding,wasysym,subcaption}
\usepackage{hyperref}
\usepackage{todonotes}
\usepackage[nameinlink,capitalise,noabbrev]{cleveref}
\usepackage[margin=1in,marginparwidth=0.8in, marginparsep=0.1in]{geometry}
\usetikzlibrary{calc,arrows}

\numberwithin{equation}{section}

\newtheorem{thm}{Theorem}[section]

\newtheorem{cor}[thm]{Corollary}
\newtheorem{lma}[thm]{Lemma}
\newtheorem{cnj}[thm]{Conjecture}

\theoremstyle{definition}
\newtheorem{dfn}[thm]{Definition}

\theoremstyle{remark}
\newtheorem{rmk}[thm]{Remark}

\numberwithin{equation}{section}

\newcommand{\R}{{\mathbb{R}}}
\newcommand{\Z}{{\mathbb{Z}}}

\DeclareMathOperator{\Perf}{Perf}
\DeclareMathOperator{\Core}{Core}
\DeclareMathOperator{\id}{id}

\DeclareMathOperator{\Aug}{\mathcal{A}ug}
\DeclareMathOperator*{\colim}{colim}

\newcommand{\suchthat}{\;\ifnum\currentgrouptype=16 \middle\fi|\;}

\begin{document}

\title{Singular Legendrian unknot links and relative Ginzburg algebras}
\author{Johan Asplund}
\address{Department of Mathematics, Stony Brook University, 100 Nicolls Road, Stony Brook, NY 11794}
\email{johan.asplund@stonybrook.edu}

\begin{abstract}
	We associate to a quiver and a subquiver $(Q,F)$ a stopped Weinstein manifold $X$ whose Legendrian attaching link is a singular Legendrian unknot link $\varLambda$. We prove that the relative Ginzburg algebra of $(Q,F)$ is quasi-isomorphic to the Chekanov--Eliashberg dg-algebra of $\varLambda$. It follows that the Chekanov--Eliashberg dg-algebra of $\varLambda$ relative to its boundary dg-subalgebra, and the Orlov functor associated to the partially wrapped Fukaya category of $X$ both admit a strong relative smooth Calabi--Yau structure.
\end{abstract}

\maketitle
\tableofcontents
\section{Introduction}\label{sec:intro}
	Let $Q$ be a quiver and let $F \subset Q$ be a subquiver called the \emph{frozen subquiver}. Define $Z_0(Q)$ to be the subcritical Weinstein manifold obtained by taking Weinstein connected sum of copies of $B^{2n}$ equipped with the standard Liouville form $\lambda = \frac 12 \sum_{i=1}^n (x_idy_i - y_idx_i)$ according to the arrows of $Q$. We define a singular Legendrian $\varLambda(Q,F)$ in the contact boundary of $Z_0(Q)$ by first taking a copy of the $(n-1)$-dimensional Legendrian unknot contained in a Darboux chart in the boundary of each copy of $B^{2n}$. Next, for each arrow in $Q$ we push part of the Legendrian unknot associated to its tail through the Weinstein $1$-handle to the copy of $B^{2n}$ associated with the head of the arrow. If the arrow is non-frozen we make the two Legendrians link once and if the arrow is frozen we make them intersect transversely once. The result is a singular Legendrian unknot link in the boundary of $Z_0(Q)$, see \cref{fig:singular_leg_unknot_link}. The pair $(Z_0(Q),\varLambda(Q,F))$ specifies a stopped Weinstein manifold $X(Q,F)$ by attaching top Weinstein handles to non-frozen components of $\varLambda(Q,F)$ and a stop to each frozen component.

	\begin{figure}[!htb]
		\centering
		\includegraphics[scale=1.5]{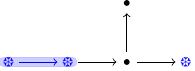}
		\hspace{1cm}
		\includegraphics[scale=0.75]{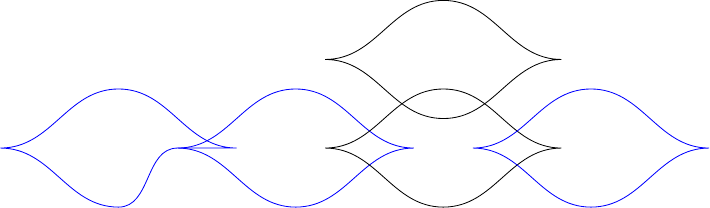}
		\caption{Left: A quiver $Q$ and frozen subquiver $F$ marked with blue snowflakes {\color{blue}\SnowflakeChevronBold} and blue arrows. Right: Front projection of corresponding singular Legendrian unknot link $\varLambda(Q,F)$ with frozen components marked in blue.}\label{fig:singular_leg_unknot_link}
	\end{figure}

	We let $CE^\ast(\varLambda(Q,F);Z_0(Q))$ denote the Chekanov--Eliashberg dg-algebra of $\varLambda(Q,F)$ over a field, which is defined as the Chekanov--Eliashberg dg-algebra after taking a Weinstein neighborhood of each frozen component of $\varLambda(Q,F)$, see \cite[Section 3.1]{asplund2020chekanov}. There is a canonical geometrically defined dg-subalgebra $\mathcal B \hookrightarrow CE^\ast(\varLambda(Q,F);Z_0(Q))$ that in particular is empty if $F = \varnothing$. Let $\mathscr G^\ast_n(Q,F)$ denote the $n$-dimensional relative Ginzburg algebra of $(Q,F)$ (see \cref{dfn:rel_ginzburg}). There is a map of dg-algebras $G \colon \mathscr G_{n-1}^\ast(F) \to \mathscr G_n^\ast(Q,F)$ (see \cref{lma:inclusion_rel_ginz}), where $\mathscr G_{n-1}^\ast(F) := \mathscr G_{n-1}^\ast(F,\varnothing)$.
	\begin{thm}\label{thm:main}
		Let $n\geq 4$ and consider a quiver $Q$ together with a subquiver $F \subset Q$. There are quasi-isomorphisms of dg-algebras $CE^\ast(\varLambda(Q,F);Z_0(Q)) \cong \mathscr G^\ast_n(Q,F)$ and $\mathcal B \cong \mathscr G^\ast_{n-1}(F)$, such that the following diagram commutes
		\[
			\begin{tikzcd}[row sep=scriptsize, column sep=scriptsize]
				\mathcal B \rar[hook]{\text{incl.}} \dar{\cong} & CE^\ast(\varLambda(Q,F);Z_0(Q)) \dar{\cong} \\
				\mathscr G^{\ast}_{n-1}(F) \rar{G} & \mathscr G^\ast_n(Q,F)
			\end{tikzcd}.
		\]
	\end{thm}
	\begin{rmk}
		\begin{enumerate}
			\item \cref{thm:main} also holds for $n = 3$ if every vertex of $F$ has valency less than three.
			\item If $F_1 = \varnothing$, then $\mathcal B \cong C_{-\ast}(\varOmega \varLambda(F))$ where $\varLambda(F) \subset \varLambda(Q,F)$ is the Legendrian sublink consisting of the frozen components only, and the dg-algebra $CE^\ast(\varLambda(Q,F);Z_0(Q))$ is the Chekanov--Eliashberg dg-algebra of $\varLambda(Q,F)$ with loop space coefficients in $C_{-\ast}(\varOmega \varLambda(F))$ as defined by Ekholm--Lekili \cite[Section 3.2]{ekholm2017duality}.
			\item \cref{thm:main} was proven in the case of $F = \varnothing$ in \cite[Corollary 4.16]{asplund2021simplicial}.
		\end{enumerate}
	\end{rmk}
	Brav--Dyckerhoff introduced the notion of a (strong) relative Calabi--Yau structure \cite{brav2019relative}. In their terminology a \emph{left} Calabi--Yau structure is a a \emph{smooth} Calabi--Yau structure in the terminology of Kontsevich--Soibelman \cite{kontsevich2009notes}. A relative smooth Calabi--Yau structure is called \emph{weak} if the quasi-isomorphisms in \cite[(1.13)]{brav2019relative} does not necessarily come from a class in relative negative cyclic homology, and \emph{strong} if it does.

	\begin{cor}\label{cor:incl_ce}
		The canonical inclusion $\mathcal B \hookrightarrow CE^\ast(\varLambda(Q,F);Z_0(Q))$ admits a strong relative smooth $n$-Calabi--Yau structure.
	\end{cor}
	\begin{proof}
		The map $G \colon \mathscr G^\ast_{n-1}(F) \to \mathscr G^\ast_n(Q, F)$ (see \cref{lma:inclusion_rel_ginz}) admits a strong relative smooth $n$-Calabi--Yau structure by \cref{lma:ginzburg_rel_cy}. Apply Theorem 1.1.
	\end{proof}
	Let $V(Q,F)$ denote a Weinstein neighborhood of the singular Legendrian $\varLambda(Q,F) \subset \partial Z_0(Q)$. It follows from \cite[Theorem 1.1]{asplund2020chekanov} that there is a derived equivalence
	\[
		\Perf CE^\ast(\varLambda(Q,F);Z_0(Q)) \cong \mathcal W(X(Q,F);V(Q,F)),
	\]
	where $\mathcal W(X(Q,F))$ denotes the partially wrapped Fukaya category of $X(Q,F)$ stopped at the Weinstein hypersurface $V(Q,F)$ and where $\Perf$ denotes the category of perfect modules, see \cite{sylvan2016partially,ganatra2017covariantly}. The Orlov functor $\iota_{V(Q,F)} \colon \mathcal W(V(Q,F)) \hookrightarrow \mathcal W(X(Q,F);V(Q,F))$ (see \cite[Section 2.4]{sylvan2019orlov} for the definition) fits in a commutative diagram
	\begin{equation}\label{eq:dia_orlov}
		\begin{tikzcd}[row sep=scriptsize, column sep=scriptsize]
			\Perf \mathcal B \rar[hook] \dar{\cong} & \Perf CE^\ast(\varLambda(Q,F);Z_0(Q)) \dar{\cong} \\
			\mathcal W(V(Q,F)) \rar[hook]{\iota_{V(Q,F)}} & \mathcal W(X(Q,F);V(Q,F))
		\end{tikzcd},
	\end{equation}
	leading to the following.
	\begin{cor}\label{cor:orlov}
		The Orlov functor $\iota_{V(Q,F)}$ admits a strong relative smooth $n$-Calabi--Yau structure.
	\end{cor}
	\subsection{Related results}
	As far as we know, the novelty of \cref{cor:incl_ce} is the existence of a \emph{strong} relative smooth Calabi--Yau structure. Quite a bit is known in the non-relative case or in the weak case, as we list below.
	\begin{itemize}
		\item In the case $F = \varnothing$ it abstractly follows from the Bourgeois--Ekholm--Eliashberg surgery isomorphism \cite{bourgeois2012effect,ekholm2017duality,ekholm2019holomorphic} together with the construction of a strong smooth $n$-Calabi--Yau structure on the wrapped Fukaya category due to Ganatra \cite[Theorem 1.3; Theorem 1.16]{ganatra2012symplectic,ganatra2019cyclic}.

		Another proof of this fact goes via \cite[Corollary 4.16]{asplund2021simplicial} and the fact that the (non-relative) Ginzburg algebra is a strong smooth $n$-Calabi--Yau algebra \cite[Theorem 6.3]{keller2011deformed}.
		\item Assume $Q = \bullet$, $F = \varnothing$ and let $\varLambda$ be a horizontally displaceable Legendrian sphere (not necessarily an unknot) in the contactization of a Liouville domain. Then Legout \cite{legout2023calabi} proved that the Chekanov--Eliashberg dg-algebra of $\varLambda$ admits a \emph{weak} $n$-Calabi--Yau structure via an explicit geometric construction.

		If $\varLambda \subset \R^3$ is not necessarily an unknot, Chen proves in forthcoming work that the augmentation category $\Aug_-(\varLambda;\Z_2)$ defined by Bourgeois--Chantraine \cite[Section 3.2]{bourgeois2014bilinearized} carries a strong proper 2-Calabi--Yau structure. In the terminology of Brav--Dyckerhoff \cite{brav2019relative} this is a \emph{right} Calabi--Yau structure.
		\item Assume $Q = F = \bullet$, and let $\varLambda \subset \R^3$ not necessarily an unknot. In forthcoming work by Ma--Sabloff it is proven that the $A_\infty$-functor $\Aug_+(\varLambda; \Z_2) \to \mathcal C(\varLambda;\Z_2)$ that is defined on morphisms as precomposition with the canonical inclusion $C_{-\ast}(\varOmega \varLambda) \hookrightarrow CE^\ast(\varLambda;C_{-\ast}(\varOmega \varLambda))$ carries a weak relative proper $2$-Calabi--Yau structure, where $\Aug_+$ denotes the augmentation category as defined in \cite[Section 4]{ng2020augmentations}.
		\item For general subcritical Weinstein manifolds and for general Legendrian submanifolds, \cref{cor:incl_ce} is proven in the weak case in forthcoming work by Dimitroglou Rizell--Legout via an explicit geometric construction
	\end{itemize}

	\cref{cor:orlov} is not new, but our proof is. It follows in the general case of stopped Weinstein manifolds admitting a stable polarization (which is always true for $X(Q,F)$) by combining work of Shende--Takeda \cite[Theorem 1.14]{shende2016calabi} and Ganatra--Pardon--Shende \cite[Theorem 1.4]{ganatra2018microlocal}.
	\subsection{Outline of the proof}
		The proof of \cref{thm:main} is a generalization of the computation in \cite[Section 4.3]{asplund2021simplicial}, where the Chekanov--Eliashberg dg-algebra of the Legendrian attaching link of a plumbing of copies of $T^\ast S^n$ for $n\geq 3$ was computed, which corresponds to the case $F = \varnothing$. A rough sketch of the proof of \cref{thm:main} is as follows.
		\begin{enumerate}
			\item Find a $1$-dimensional \emph{simplicial decomposition} of the subcritical Weinstein manifold $Z_0(Q)$ in the sense of \cite[Section 2.2.4]{asplund2021simplicial}.
			\item Find a singular Legendrian $\boldsymbol \varLambda(Q,F)$ relative to the simplicial decomposition of $Z_0(Q)$ whose completion is the singular Legendrian $\varLambda(Q,F)$ (see \cite[Definitions 4.1 and 4.2]{asplund2021simplicial}).
			\item Compute the Chekanov--Eliashberg dg-algebra of each part of $\boldsymbol \varLambda(Q,F)$. Each such piece is either a Legendrian unknot with boundary, a Legendrian Hopf link with boundary or two Legendrian unknots with a transverse intersection, with boundary.
			\item Using the local computation and the gluing formula for Chekanov--Eliashberg dg-algebras of singular Legendrians \cite[Theorem 2.43]{asplund2022tangle} we recover $CE^\ast(\varLambda(Q,F); Z_0(Q))$ up to quasi-isomorphism.
			\item Write down an explicit chain homotopy equivalence between the resulting Chekanov--Eliashberg dg-algebra and the relative Ginzburg algebra of $(Q,F)$.
		\end{enumerate}
		\subsection{Future directions}
			Even though the relative Ginzburg algebra is known to not be a strong smooth Calabi--Yau algebra, it is expected to admit a pre-Calabi--Yau structure in the sense of Kontsevich--Takeda--Vlassopoulos \cite{kontsevich2021pre}. This is consistent with the folklore conjecture that the partially wrapped Fukaya category should admit a pre-Calabi--Yau structure, which is motivated by work of Seidel \cite[Section 3.3]{seidel2012fukaya}. In view of the surgery formula for partially wrapped Fukaya categories \cite[Theorem 1.1]{asplund2020chekanov} we thus have the following.
			\begin{cnj}
				The Chekanov--Eliashberg dg-algebra of a singular Legendrian submanifold in the boundary of a stopped Weinstein manifold admits a geometrically defined pre-Calabi--Yau structure.
			\end{cnj}
			Recently Ng \cite{ng2023} defined an $L_\infty$-structure on a commutative version of the Chekanov--Eliashberg dg-algebra of a Legendrian in $\R^3$ with loop space coefficients, which we expect to be a shadow of the conjectured pre-Calabi--Yau structure.

			Throughout this paper, every quiver come equipped with the trivial potential. An interesting problem that we hope to return to in the future is realizing the Ginzburg algebra of any $(Q,F)$ equipped with a non-trivial potential as a Chekanov--Eliashberg dg-algebra. The only example of such a relation known to date is due to Li \cite[Theorem 1.2]{li2019koszul}.
	\subsection*{Acknowledgments}
		The author is grateful to Georgios Dimitroglou Rizell and Noémie Legout for helpful comments on a draft of this paper and for explaining their work in progress, Merlin Christ for his guidance in the literature on relative Ginzburg algebras, Lenhard Ng for sharing a draft of recent work \cite{ng2023}, Joshua Sabloff for sharing his forthcoming results with Jiajie Ma and Zhenyi Chen for sharing and explaining his forthcoming results. Finally the author thanks the anonymous referee for valuable suggestions, improving the exposition of the paper. The author was supported by the Knut and Alice Wallenberg Foundation.
\section{Relative Ginzburg algebras}
	Let $\boldsymbol k$ be a field. Let $Q$ be a quiver (with trivial potential), and let $F \subset Q$ be a subquiver. Let $\overline Q_n(F)$ denote the graded quiver with vertex set $Q_0$ and arrow set $Q_1$ consisting of the following.
	\begin{itemize}
		\item An arrow $g \colon v\to w$ in degree $0$ for each $g \colon v \to w$ in $Q_1$.
		\item An arrow $g^\ast \colon w \to v$ in degree $2-n$ for each $g \colon v \to w$ in $Q_1 \smallsetminus F_1$.
		\item An arrow $h_v \colon v \to v$ in degree $1-n$ for each $v\in Q_0 \smallsetminus F_0$.
	\end{itemize}
	We will use the notation $\overline Q_n := \overline Q_n(\varnothing)$ and
	\[
		[x,y]^A_v := \sum_{A \ni e \colon v \to \bullet} xy - \sum_{A \ni e \colon \bullet \to v} yx ,
	\]
	where $A \subset Q_1$ is a subset.
	\begin{dfn}[$n$-dimensional Ginzburg algebra]
		The \emph{$n$-dimensional Ginzburg algebra} $\mathscr G^\ast_n(Q)$ of the quiver $Q$ is the path algebra of $\overline Q_n$ equipped with the differential $d$ given on arrows by
		\[
			dg = dg^\ast = 0, \quad dh_v = [g,g^\ast]^{Q_1}_v, \quad v \in Q_0,
		\]
		and extended to the whole path algebra by linearity and the Leibniz rule.
	\end{dfn}
	\begin{dfn}[$n$-dimensional relative Ginzburg algebra]\label{dfn:rel_ginzburg}
		The \emph{$n$-dimensional relative Ginzburg algebra} $\mathscr G^\ast_n(Q,F)$ of the tuple $(Q,F)$ is the path algebra of $\overline Q_n(F)$ equipped with the differential $d$ given on arrows by
		\[
			dg = dg^\ast = 0, \quad dh_v = [g,g^\ast]^{Q_1 \smallsetminus F_1}_v, \quad v \in Q_0 \smallsetminus F_0,
		\]
		and extended to the whole path algebra by linearity and the Leibniz rule.
	\end{dfn}
	\begin{lma}\label{lma:inclusion_rel_ginz}
		There is a map of dg-algebras $G \colon \mathscr G^\ast_{n-1}(F) \to \mathscr G^\ast_n(Q,F)$ that is defined on generators as $g \mapsto g$, $g^\ast \mapsto 0$, $h_v \mapsto [g,g^\ast]^{Q_1 \smallsetminus F_1}_v$ and extended to the whole of $\mathscr G^\ast_{n-1}(F)$ by linearity and multiplicativity.
	\end{lma}
	\begin{proof}
		This is immediate from the definitions.
	\end{proof}
	The following result seems well-known to experts and implicitly alluded to in \cite[Section 7.2]{wu2023relative} and \cite[Section 8]{keller2021introduction} but we did not find an explicit reference in the literature.
	\begin{lma}\label{lma:ginzburg_rel_cy}
		The dg-algebra map $G \colon \mathscr G^\ast_{n-1}(F) \to \mathscr G^\ast_n(Q,F)$ defined in \cref{lma:inclusion_rel_ginz} admits a strong relative smooth $n$-Calabi--Yau structure.
	\end{lma}
	\begin{proof}
		It is well-known that $\mathscr G^\ast_{n-1}(F) \cong \varPi_n(\boldsymbol kF)$ by \cite[Theorem 6.3]{keller2011deformed}, where $\varPi_n(\boldsymbol kF)$ denotes the $n$-Calabi--Yau completion of the path algebra of $F$. Yeung defined the relative $n$-Calabi--Yau completion $\varPi_n(\boldsymbol kQ, \boldsymbol kF)$ in \cite[Section 7]{yeung2016relative}, and Wu defined a reduced version $\varPi_n^{\text{red}}(\boldsymbol kQ, \boldsymbol kF)$ in \cite[Section 3.6]{wu2023relative}, which is quasi-isomorphic to $\varPi_n(\boldsymbol kQ, \boldsymbol kF)$ by \cite[Proposition 3.18]{wu2023relative}. It was proven by Yeung \cite[Theorem 7.1]{yeung2016relative}, Bozec--Calaque--Scherotzke \cite[Corollary 5.24]{bozec2020relative} and Wu \cite[Proposition 3.18]{wu2023relative} that the inclusion $\varPi_n(\boldsymbol kF) \hookrightarrow \varPi_n(\boldsymbol kQ, \boldsymbol kF)$ admits a strong relative smooth $n$-Calabi--Yau structure.

		Thus it suffices to show that $\varPi_n^{\text{red}}(\boldsymbol kQ, \boldsymbol kF) \cong \mathscr G_n^\ast(Q,F)$, which is well-known to experts. In the same way that the $n$-Calabi--Yau completion in \cite{keller2011deformed} is defined as the tensor algebra of a shift of the inverse dualizing bimodule, the relative $n$-Calabi--Yau completion is defined as the tensor algebra of a shift of the \emph{relative} inverse dualizing bimodule. An explicit description of the reduced version is described in \cite[p.\@ 37]{wu2023relative}, and the underlying module is generated by the dual arrows $g^\ast$ in degree $2-n$ for every $g \in Q_1 \smallsetminus F_1$ and loops $h_v$ in degree $1-n$ for every $v \in Q_0 \smallsetminus F_0$. A similar description for $n = 2$ is found in \cite[Section 5.3.2]{bozec2020relative} but holds for general $n$ in a similar manner. This finishes the proof.
	\end{proof}	
	\begin{rmk}
		The dg-algebra $\varPi_n(\boldsymbol kQ, \boldsymbol kF)$ was also defined in \cite{keller2021introduction} under the name \emph{relative derived preprojective algebra}.
	\end{rmk}
\section{Stopped Weinstein manifold associated to a quiver and a frozen subquiver}
	\subsection{Weinstein manifolds and stops}
		Let us briefly recall some basic geometric definitions.
		\begin{dfn}[Weinstein hypersurface]\label{dfn:weinstein_hypersurface}
			A \emph{Weinstein hypersurface} is a Weinstein embedding $(V^{2n-2}, \lambda|_V) \hookrightarrow (X \smallsetminus \Core X, \lambda)$ such that the induced map $V \to \partial X$ is an embedding. We will denote Weinstein embeddings by $V \hookrightarrow \partial X$.
		\end{dfn}
		\begin{dfn}[Weinstein pair]
			A \emph{Weinstein pair} is a tuple $(X^{2n},V^{2n-2})$ consisting of a Weinstein manifold $X$ and a Weinstein hypersurface $V \hookrightarrow \partial X$.
		\end{dfn}
		\begin{dfn}[Stop associated to a Weinstein hypersurface]
			Let $(X,V)$ be a Weinstein pair. The \emph{stop associated to the Weinstein hypersurface $V \hookrightarrow \partial X$} is the Weinstein cobordism $\sigma_V$ obtained by gluing the Weinstein cobordism $(V \times D_\varepsilon T^\ast \varDelta^1, \lambda_V + 2xdy+ydx)$ to $\R \times (V \times \R)$ along $V \times (-\varepsilon,\varepsilon) \hookrightarrow \partial(\R \times (V \times \R))$.
		\end{dfn}
		\begin{dfn}[Weinstein manifold stopped at a Weinstein hypersurface]\label{dfn:stopped_weinstein_manifold}
			Let $(X,V)$ be a Weinstein pair. The \emph{Weinstein manifold $X$ stopped at $V$} is the Weinstein cobordism obtained by gluing $\sigma_V$ to $X$ along $V \times (-\varepsilon,\varepsilon)$.
		\end{dfn}
		\begin{rmk}
			What we call a ``stop'' deviates slightly from the original definition in \cite{sylvan2016partially}. Our usage of the word stop originates from its usage in \cite{ekholm2017duality,asplund2021fiber,asplund2020chekanov}.
		\end{rmk}
		\begin{dfn}[Weinstein manifold stopped at a singular Legendrian]\label{dfn:stopped_weinstein_manifold_leg}
			Let $X$ be a Weinstein manifold and $\varLambda \subset \partial X$ a singular Legendrian. We define the \emph{Weinstein manifold $X$ stopped at $\varLambda$} to be the Weinstein manifold $X$ stopped at a Weinstein neighborhood $T^\ast \varLambda$ of $\varLambda$.
		\end{dfn}
		\begin{dfn}[Legendrian relative to Weinstein pair]\label{dfn:leg_rel}
			Let $(X,V)$ be a Weinstein pair. A Legendrian submanifold relative to $V$ is an embedded Legendrian submanifold-with-boundary $\varLambda \subset \partial X$ such that $\partial \varLambda \subset \partial V$ is Legendrian.
		\end{dfn}
	\subsection{A simplicial decomposition}\label{sec:simpl_decomp}
		Let $Q$ be a quiver and let $F \subset Q$ be a subquiver called the \emph{frozen subquiver}. Recall the definition of the pair $(Z_0(Q),\varLambda(Q,F))$ in \cref{sec:intro}. We pick a Weinstein neighborhood of each frozen component of $\varLambda(Q,F)$ together with a handle decomposition that consists of one Weinstein $0$-handle per frozen vertex of the component and one Weinstein $1$-handle per frozen edge of the component. Then let $P_0(F)$ denote the union of the subcritical parts of each of these Weinstein neighborhoods. The purpose of this section is to describe a simplicial decomposition of $(Z_0(Q),P_0(F))$ (see \cite[Definition 2.15]{asplund2021simplicial}) and a Legendrian submanifold $\boldsymbol \varLambda(Q,F)$ relative to this simplicial decomposition (see \cite[Definition 2.38]{asplund2022tangle}). The construction closely follows \cite[Section 4.3]{asplund2021simplicial}.

		Let $C$ be the graph that is obtained by taking the underlying graph of $Q$, and adding one extra vertex in the middle of each edge (thus splitting the edge into two). For convenience, we consider a partition $C_0 = V \cup F_0 \cup E \cup F_1$ where $V$ corresponds to elements of $Q_0 \smallsetminus F_0$ (non-frozen vertices) and $E$ corresponds to elements in $Q_1 \smallsetminus F_1$. The set $E \cup F_1$ consists of the newly added vertices (edge vertices), see \cref{fig:graph_simpl_complex}.
		\begin{figure}[!htb]
			\centering
			\includegraphics[scale=1.3]{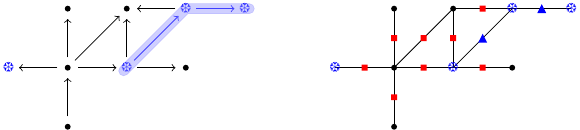}

			\vspace{5mm}

			\includegraphics[scale=1.3]{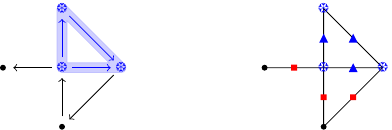}
			\caption{Left: Plumbing quivers $Q$ with a set $F_0$ of frozen vertices and a set $F_1$ of frozen edges. Right: The corresponding graph $C$ with vertex set $V \cup F_0 \cup E \cup F_1$. The vertices represented by black dots $\bullet$ belong to $V$, the vertices represented by blue snowflakes {\color{blue}\SnowflakeChevronBold} belong to $F_0$, the vertices represented by red squares $\textcolor{red}{\blacksquare}$ belong to $E$ and the vertices represented by blue triangles {\color{blue}\UParrow} belong to $F_1$.}\label{fig:graph_simpl_complex}
		\end{figure}

		A simplicial decomposition of $(Z_0(Q),P_0(F))$ is a tuple $((C,C'),(\boldsymbol V, \boldsymbol V'))$ with $C$ as above and $C' = F_0 \cup F_1$ (where $F_0 \cup F_1$ is viewed as a subset of the vertex set $C_0$). We call $e \in C_1$ a \emph{frozen edge} if $e$ is adjacent to any frozen edge vertex. All other edges are called non-frozen.

		Both $\boldsymbol V$ and $\boldsymbol V'$ are sets consisting of Weinstein manifolds and Weinstein hypersurfaces of varying dimensions. For simplicity we denote by $B^{2n}_{\text{std}}$ the Weinstein manifold consisting of $B^{2n}$ equipped with its standard Liouville one-form $\lambda = \frac 12 \sum_{i=1}^{n} (x_i dy_i - y_i dx_i)$.
		The set $\boldsymbol V$ consists of
		\begin{enumerate}
			\item one copy of $V^{2n-2}_{e} := B^{2n-2}_{\text{std}}$ for each edge $e \in C_1$;
			\item one copy of $V^{2n}_v := B^{2n}_{\text{std}}$ for each vertex $v \in C_0$;
			\item one Weinstein hypersurface $\iota_v\colon\bigsqcup_{C_1 \ni e \supset v} V^{2n-1}_e \hookrightarrow \partial V^{2n}_v$ for each vertex $v \in C_0$.
		\end{enumerate}
		The set $\boldsymbol V'$ consists of 
		\begin{enumerate}
			\item one copy of $(V')^{2n-2}_v := B^{2n-2}_{\text{std}}$ for each frozen vertex and for each frozen edge vertex $v \in F_0 \cup F_1 \subset C_0$;
			\item one copy of $(V')^{2n-4}_{e} := B^{2n-4}_{\text{std}}$ for each frozen edge  $e\in F_1$;
			\item one Weinstein hypersurface $\bigsqcup_{C_1 \ni e \supset v} (V')^{2n-4}_{e} \hookrightarrow \partial (V')^{2n-2}_v$ for each frozen vertex $v \in F_0$ and frozen edge vertex $v \in F_1$.
		\end{enumerate}
		We have that $(C,\boldsymbol V)$ is a simplicial decomposition of $Z_0(Q)$, and $(C', \boldsymbol V')$ is a simplicial decomposition of $P_0$. To complete the construction, we have a hypersurface inclusion $(C', \boldsymbol V') \hookrightarrow (C, \boldsymbol V)$ (see \cite[Definition 2.12]{asplund2022tangle}) which consists of
		\begin{enumerate}
			\item the inclusion $C' = F \hookrightarrow C$ of the (underlying graph of the) frozen subquiver into $C$;
			\item one Weinstein hypersurface $(V')^{2n-4}_{e} \hookrightarrow V^{2n-2}_e$ for each edge $e\in C_1$;
			\item one Weinstein hypersurface
			\[
				(V')_v^{2n-2} \#_{(V')_{e}^{2n-4}} V_e^{2n-2} \hookrightarrow \partial V_v^{2n}
			\]
			for each pair of a frozen (edge) vertex $v \in F_0 \cup F_1$ and an edge $e \in C_1$ adjacent to $v$.
		\end{enumerate}
		An example of the simplicial decomposition of $(Z_0(Q),P_0(F))$ is depicted in \cref{fig:simp_decomp_pair}.
		\begin{figure}[!htb]
			\centering
			\includegraphics[scale=1.5]{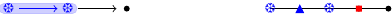}

			\vspace{5mm}

			\includegraphics[width=0.9\textwidth]{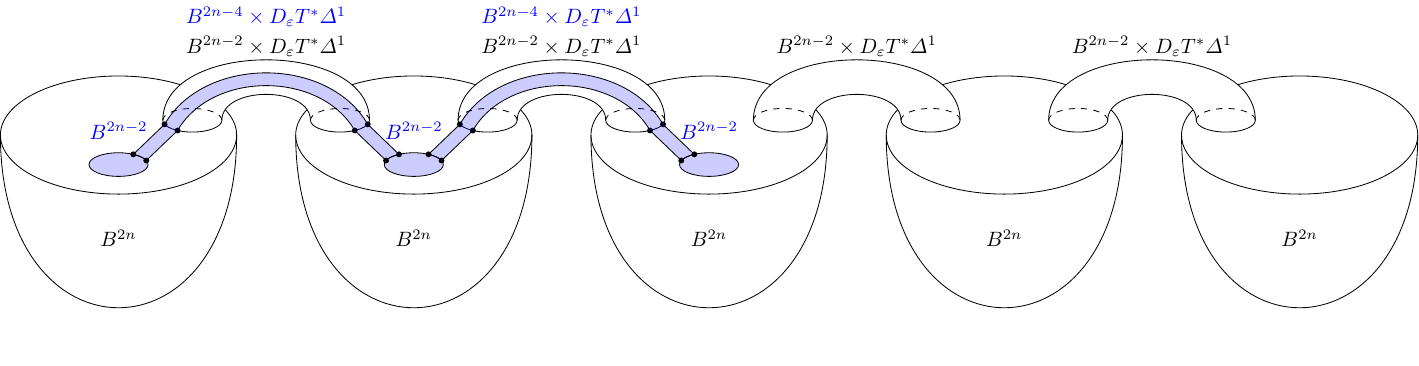}
			\caption{Top left: A quiver with a frozen subquiver $(Q,F)$. Top right: The simplicial complex $C$ associated to $(Q,F)$. Bottom: The simplicial decomposition of the pair $(Z_0(Q),P_0(F))$.}\label{fig:simp_decomp_pair}
		\end{figure}
		We now define a Legendrian submanifold relative to $((C',C), (\boldsymbol V,\boldsymbol V'))$ denoted by $\boldsymbol \varLambda(Q,F)$ (see \cite[Definition 2.38]{asplund2022tangle}) as follows:
		\begin{enumerate}
			\item For each non-frozen edge vertex $v\in E$ let $\varLambda_v$ be the $(n-1)$-dimensional Hopf link $\varLambda_1 \cup \varLambda_2$ with boundary such that $\partial \varLambda_i \subset \partial V^{2n-2}_{e_i}$ is the $(n-2)$-dimensional Legendrian unknot, where $e_1,e_2 \in C_1$ are the two edges adjacent to $v$, see \cref{fig:locally_at_edge_vertex}.
			\item For each frozen edge vertex $v\in F_1$ let $\varLambda_v$ be two $(n-1)$-dimensional Legendrian disks $\varLambda_1 \cup \varLambda_2$ with boundary in $\partial(V^{2n-2}_{e_1} \#_{(V')^{2n-4}_{e_1}} (V')^{2n-2}_v \#_{(V')^{2n-4}_{e_2}} V^{2n-2}_{e_2})$ such that
			\begin{enumerate}
				\item $\partial \varLambda_i \cap \partial V^{2n-2}_{e_i} \subset \partial V^{2n-2}_{e_i}$ is the $(n-2)$-dimensional Legendrian unknot with boundary a $(n-3)$-dimensional Legendrian unknot in $\partial (V')^{2n-2}_{e_i}$, where $e_1,e_2 \in C_1$ are the two edges adjacent to $v$, and
				\item $(\partial \varLambda_1 \cup \partial \varLambda_2) \cap \partial (V')^{2n-2}_{v} \subset \partial (V')^{2n-2}_v$ is the $(n-2)$-dimensional Hopf link with two boundary components, one in each $\partial (V')^{2n-4}_{e_i}$ looking like the $(n-3)$-dimensional unknot, where $e_1,e_2 \in C_1$ are the two edges adjacent to $v$
			\end{enumerate}
			see \cref{fig:locally_at_frozen_edge_vertex}.
			\item For each $v \in V$, let $\varLambda_v$ be the $(n-1)$-dimensional Legendrian unknot with $k$ open disks removed, where $k$ is the valency of the vertex $v$ such that each of the $k$ components of $\partial \varLambda_v$ is the $(n-2)$-dimensional Legendrian unknot in $\partial V^{2n-2}_{e_i}$ where $e_1,\ldots,e_k \in C_1$ are the $k$ edges that are adjacent to $v$, see \cref{fig:locally_at_vertex}.
			\item For each $v \in F$, let $\varLambda_v$ is the $(n-1)$-dimensional Legendrian unknot with one boundary component being a $(n-2)$-dimensional Legendrian unknot in $\partial B^{2n-2}$ for each non-frozen edge adjacent to $v$. 

			The new copy of $B^{2n-2}$ coming from the fact that $v \in F$ (see the definition of $\boldsymbol V'$) is connected via a Weinstein $1$-handle to every other $B^{2n-2}$ corresponding to frozen edges of $C$ that are adjacent to $v$. The boundary $\varLambda_v$ in this new $B^{2n-2}$ is a $(n-2)$-dimensional unknot with some number of disks removed (corresponding to the number of frozen edges adjacent to $v$), see \cref{fig:locally_at_frozen_vertex}.
		\end{enumerate}
		\begin{figure}[!htb]
			\centering
			\includegraphics[scale=1.5]{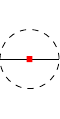}
			\hspace{15mm}
			\includegraphics{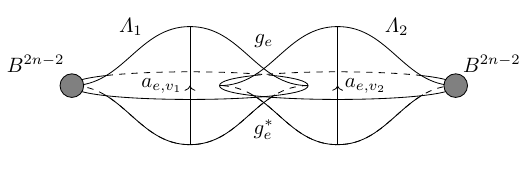}
			\includegraphics{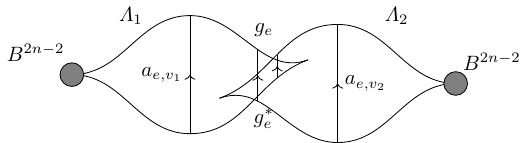}
			\caption{Top left: Local depiction of a vertex $v\in E$. Top right: Two copies of $B^{2n-2}$ with $\varLambda_{v}$ which becomes the $(n-1)$-dimensional Hopf link with boundary in $\partial(B^{2n-2} \sqcup B^{2n-2})$ after a small perturbation in a Darboux chart. Bottom: A slice of $\varLambda_{v}$ in generic position after a small perturbation, where the Reeb chords $g_e$ and $g_e^\ast$ are visible.}\label{fig:locally_at_edge_vertex}
		\end{figure}
		\begin{figure}[!htb]
			\centering
			\includegraphics[scale=1.5]{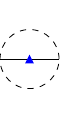}
			\hspace{15mm}
			\includegraphics{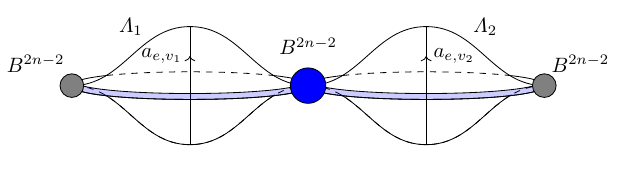}

			\includegraphics{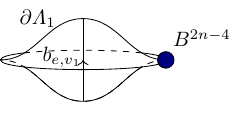}
			\hspace{2mm}
			\includegraphics{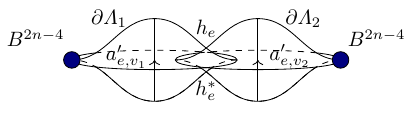}
			\hspace{2mm}
			\includegraphics{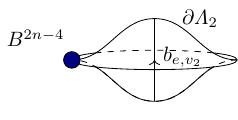}
			\caption{Top left: Local depiction of a vertex $v\in F_1$. Top right: Three copies of $B^{2n-2}$ with $\varLambda_{v}$ being two $(n-1)$-dimensional Legendrian disks with boundary in $\partial(B^{2n-2} \#_{B^{2n-4}} B^{2n-2} \#_{B^{2n-4}} B^{2n-2})$ in a Darboux chart. Bottom left: The boundary of $\varLambda_1$ in the leftmost copy of $B^{2n-2}$. Bottom middle: The boundary of $\varLambda_1 \cup \varLambda_2$ in the middle blue copy of $B^{2n-2}$. Bottom right: The boundary of $\varLambda_2$ in the rightmost copy of $B^{2n-2}$.}\label{fig:locally_at_frozen_edge_vertex}
		\end{figure}
		\begin{figure}[!htb]
			\centering
			\includegraphics[scale=1.5]{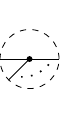}
			\hspace{30mm}
			\includegraphics{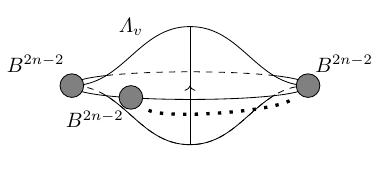}
			\caption{Left: Local depiction of a vertex $v\in V$ with valency $k$. Right: $k$ copies of $B^{2n-2}$ with $\varLambda_{v}$ being the $(n-1)$-dimensional Legendrian unknot with boundary in $\partial \left(\bigsqcup_{i=1}^k B^{2n-2}\right)$ in a Darboux chart.}\label{fig:locally_at_vertex}
		\end{figure}
		\begin{figure}[!htb]
			\centering
			\includegraphics[scale=1.5]{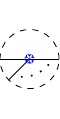}
			\hspace{30mm}
			\includegraphics{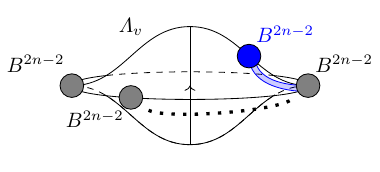}
			\caption{Left: Local depiction of a frozen vertex $v\in F$ with valency $k$. Right: $k+1$ copies of $B^{2n-2}$ with $\varLambda_{v}$ being the $(n-1)$-dimensional Legendrian unknot with boundary in $\partial \left(\bigsqcup_{\text{adjacent non-frozen edges}} B^{2n-2} \sqcup \left({\color{blue}B^{2n-2}} \#_{B^{2n-4}} B^{2n-2} \#_{B^{2n-4}} \cdots \#_{B^{2n-4}} B^{2n-2}\right)\right)$ in a Darboux chart.}\label{fig:locally_at_frozen_vertex}
		\end{figure}
		The subcritical part $Z_0(Q)$ is now constructed from $(C,\boldsymbol V)$ via gluing as described in \cite[Section 1.3]{asplund2021simplicial}, and similarly $(Z_0(Q),P_0(F))$ is constructed from $((C,C'),(\boldsymbol V,\boldsymbol V'))$ via gluing as described in \cite[Section 2.1]{asplund2022tangle}.
\section{Proof of the main theorem}
	In \cref{sec:simpl_decomp} we defined a simplicial decomposition of the subcritical Weinstein pair $(Z_0(Q),P_0(F))$, and a Legendrian $\boldsymbol \varLambda(Q,F)$ relative to the simplicial decomposition. In this section we compute the Chekanov--Eliashberg dg-algebra of $\boldsymbol \varLambda(Q,F)$ which by \cite[Definition 2.41]{asplund2022tangle} is the Chekanov--Eliashberg dg-algebra of the singular Legendrian $\varLambda(Q,F) \subset \partial Z_0(Q)$.
	\subsection{Computation of the Chekanov--Eliashberg dg-algebra}\label{sec:computation_ce_dga}
		We denote by $CE^\ast(\varLambda;Z)$ the Chekanov--Eliashberg dg-algebra of $\varLambda$ computed in the boundary of the Weinstein manifold $Z$.
		\begin{lma}\label{lma:dga_in_chart}
			Let $X_0$ be a subcritical stopped Weinstein manifold and $\varLambda \subset \partial X_0$ a Legendrian submanifold such that $\varLambda$ is contained in an arbitrarily small Darboux chart $D \subset \partial X_0$. Then there is a quasi-isomorphism $CE^\ast(\varLambda;\partial X_0) \cong CE^\ast(\varLambda;D)$.
		\end{lma}
		\begin{proof}
			This follows from a well-known energy argument and follows the same idea as in the proof of \cite[Lemma 2.37]{asplund2021simplicial} which we give an outline of here. Without loss of generality we may assume that the Darboux chart $D$ is disjoint from every subcritical Weinstein handle and stop for dimension reasons.

			Let $A,\varepsilon \in \R_{>0}$. Denote by $X_0^\varepsilon$ the subcritical stopped Weinstein manifold $X_0$ where all its subcritical handles and stops have attaching regions of size at most $\varepsilon$ (see \cite[Definition 2.11]{asplund2021simplicial}). We call any Reeb chord of $\varLambda$ that leaves $D$ external, and all other Reeb chord internal. For any external Reeb chord $c$ of $\varLambda$ of length $\ell(c) < A$ we can find $\varepsilon$ small enough so that $c$ is disjoint from every subcritical Weinstein handle and stop of $X_0^\varepsilon$ (cf.\@ \cite[Lemma 2.24]{asplund2021simplicial}). Finally we have that $J$-holomorphic disks in $\R \times \partial X_0^\varepsilon$ asymptotic to an internal Reeb chord at infinity do not escape $\R \times D$ which follows from a monotonicity argument from which we conclude the result, see \cite[Lemmas 5.9 and 5.10]{ekholm2015legendrian}.
		\end{proof}

		The main tool we use to compute the Chekanov--Eliashberg dg-algebra of $\boldsymbol \varLambda(Q,F)$ is the following gluing formula.
		\begin{lma}\label{lma:colimit_loop}
			There is a quasi-isomorphism of dg-algebras
			\[
				CE^\ast(\boldsymbol \varLambda(Q,F);(Z_0(Q),P_0(F))) \cong \colim_{x \in C_0 \cup C_1} \mathcal A_x =: \mathcal A,
			\]
			where $\mathcal A_x$ is a dg-algebra generated by the Reeb chords of $\varLambda_x$ that are completely contained in a Darboux chart, where $\varLambda_x := \partial \varLambda_v$ for $x \in C_1$ and $v \in C_0$ is any vertex adjacent to $x$ (see the definition of $\boldsymbol \varLambda(Q,F)$ in \cref{sec:simpl_decomp}). The differential on $\mathcal A_x$ is the one induced from $CE^\ast(\boldsymbol \varLambda(Q,F);(Z_0(Q),P_0(F)))$.
		\end{lma}
		\begin{proof}
			This follows from \cite[Theorem 2.43]{asplund2022tangle} which applies in our situation because $P_0(F)$ is subcritical (cf.\@ \cite[Remark 4.7]{asplund2021simplicial}) together with \cref{lma:dga_in_chart}. As remarked in the proof of \cite[Theorem 2.43]{asplund2022tangle}, the proof is almost identical to that of \cite[Theorem 2.33]{asplund2021simplicial}.
		\end{proof}
		\begin{rmk}
			The colimit in \cref{lma:colimit_loop} should be interpreted as follows: For each $x \in C_0 \cup C_1$ we have a dg-algebra $\mathcal A_x$ and for each vertex inclusion $C_0 \ni v \subset e \in C_1$ there is an inclusion of dg-algebras $\mathcal A_e \subset \mathcal A_v$, and the colimit is taken of this diagram of dg-algebras. 

			All dg-algebras are by definition semifree and hence the colimit is again a semifree dg-algebra that is generated by the union of the generators. See \cite[Remark 1.3 and Section 2.5]{asplund2021simplicial} for further details regarding coefficient rings and in which category the colimit is taken.
		\end{rmk}
		\begin{rmk}
			If $Q$ is a tree, we can by a Legendrian isotopy fit the entire $\varLambda(Q,F)$ inside a single arbitrarily small Darboux ball in $\partial Z_0(Q) \cong S^{2n-1}$. In this case the Chekanov--Eliashberg dg-algebra is computable directly without employing the gluing formula in \cref{lma:colimit_loop}.
		\end{rmk}
		We now compute each $\mathcal A_x$ appearing in the colimit in \cref{lma:colimit_loop}. The computation for $v \in V \cup E$ is identical to the computation in \cite[Section 4.3.2]{asplund2021simplicial}.
		\begin{description}
			\item[$v \in V$] The dg-subalgebra $\mathcal A_v$ for $v \in V$ is generated by $a_v$ in degree $1-n$ and $b_{e,v}$ in degree $2-n$ for each $e \in C_1$ adjacent to $v$. The differential on the generators is given by
			\[
				\partial a_v = \sum_{e \supset v} b_{e,v}, \quad \partial b_{e,v} = 0,
			\]
			where the sum is taken over all edges adjacent to $v$. The Morse flow trees counted in $\partial a_v$ are depicted in \cref{fig:flow_tree1}.
			\item[$v \in F_0$] The dg-subalgebra $\mathcal A_v$ for $v \in F_0$ is generated by $a_v$ in degree $1-n$, $b_v$ in degree $2-n$, $b_{e,v}$ in degree $2-n$ for each $e \in C_1$, and $d_{e,v}$ in degree $3-n$ for each frozen $e \in C_1$ that is adjacent to $v$. The differential on the generators is given by
			\begin{align*}
				\partial a_v = -b_v + \sum_{e \supset v} b_{e,v}, \quad \partial b_v = \sum_{\substack{e \supset v \\ \text{frozen}}} d_{e,v}, \quad \partial b_{e,v} &= \begin{cases}
					d_{e,v} & \text{$e$ frozen} \\
					0, & \text{$e$ non-frozen}
				\end{cases}, \quad \partial d_{e,v} = 0.
			\end{align*}
			where the sum is taken over all edges adjacent to $v$. The appearance of the extra term $b_v$ and the generators $d_{e,v}$ in the differential comes from the extra unknot boundary component of $\varLambda_v$ that lives in the boundary of $B^{2n-2} \#_{B^{2n-4}} \cdots \#_{B^{2n-4}} B^{2n-2}$ (see \cref{fig:locally_at_frozen_vertex}). The Morse flow trees counted in $\partial a_v$ are depicted in \cref{fig:flow_tree2,fig:flow_tree3}. The Morse flow trees counted in $\partial b_v$ and $\partial b_{e,v}$ are similar to those in \cref{fig:flow_tree1} except that they appear in the boundary of the corresponding gray or blue copy of $B^{2n-2}$.
			\item[$e \in E$] The dg-subalgebra $\mathcal A_e$ for $e \in E$ is generated by $a_{e,v_1}$ and $a_{e,v_2}$ both in degree $1-n$, $b_{e,v_1}$ and $b_{e,v_2}$ both in degree $2-n$, $g_e$ in degree $0$ and $g_e^\ast$ in degree $2-n$. By construction $e\in E$ corresponds to an arrow $e$ in $Q$, and we let $v_1,v_2 \in C_0$ denote the two vertices adjacent to $e$. The differential on the generators is given by 
			\begin{align*}
				\partial a_{e,v_1} &= \begin{cases}
					b_{e,v_1} - g_eg_e^\ast, & \text{if $e \colon v_1 \to v_2$} \\
					b_{e,v_1} + g_e^\ast g_e, & \text{if $e \colon v_2 \to v_1$}
				\end{cases}, \quad \partial a_{e,v_2} = \begin{cases}
					b_{e,v_2} + g_e^\ast g_e, & \text{if $e \colon v_1 \to v_2$} \\
					b_{e,v_2} - g_e g_e^\ast, & \text{if $e \colon v_2 \to v_1$}
				\end{cases}\\
				\partial b_{e,v_1} &= \partial b_{e,v_2} = \partial g_e = \partial g_e^\ast = 0.
			\end{align*}
			The Morse flow trees counted by $\partial a_{e,v_1}$ and $\partial a_{e,v_1}$ are depicted in \cref{fig:flow_tree4,fig:flow_tree5,fig:flow_tree6,fig:flow_tree7}.
			\item[$e \in F_1$] The dg-subalgebra $\mathcal A_e$ for $e \in F_1$ is generated by the set
			\[
				\{a_{e,v_1},a_{e,v_2},c,b_{e,v_1},b_{e,v_2},d_{e,v_1},d_{e,v_2},a_{e,v_1}',a_{e,v_2}',h_e^\ast,h_e\}
			\]
			in degrees
			\[
				|a_{e,v_i}| = 1-n, \quad |c| = |b_{e,v_i}| = |a_{e,v_i}'| = 2-n, \quad |h_e^\ast| = |d_{e,v_i}| = 3-n, \quad |h_e| = 0,
			\]
			where by construction $e \in F_1$ corresponds to a frozen arrow $e$ in $Q$ and we let $v_1,v_2 \in C_0$ denote the two vertices adjacent to $e$. The differential on the generators is given by
			\begin{align*}
				\partial a_{e,v_1} &= \begin{cases}
					b_{e,v_1} - h_ec - a_{e,v_1}', & \text{if $e \colon v_1 \to v_2$} \\
					b_{e,v_1} + ch_e - a_{e,v_1}', & \text{if $e \colon v_2 \to v_1$}
				\end{cases}, \quad \partial a_{e,v_2} = \begin{cases}
					b_{e,v_2} + c h_e - a_{e,v_2}', & \text{if $e \colon v_1 \to v_2$} \\
					b_{e,v_2} - h_e c - a_{e,v_2}', & \text{if $e \colon v_2 \to v_1$}
				\end{cases}\\
				\partial a_{e,v_1}' &= \begin{cases}
					d_{e,v_1}- h_eh_e^\ast, & \text{if $e \colon v_1 \to v_2$} \\
					d_{e,v_1} + h_e^\ast h_e, & \text{if $e \colon v_2 \to v_1$}
				\end{cases}, \quad \partial a_{e,v_2}' = \begin{cases}
					d_{e,v_2}+h_e^\ast h_e, & \text{if $e \colon v_1 \to v_2$} \\
					d_{e,v_2}- h_eh_e^\ast, & \text{if $e \colon v_2 \to v_1$}
				\end{cases}\\
				\partial c &= h_e^\ast, \quad \partial b_{e,v_1} = \partial b_{e,v_2} = \partial h_e = \partial h_e^\ast = 0.
			\end{align*}
			This dg-algebra is computed as follows. The more non-degenerate front projection of $\varLambda_1 \cup \varLambda_2$ is shown in \cref{fig:frozen_edge_local_chord} and it is the high dimensional version of the singular Legendrian $A_2$-bouquet exhibited in \cite[Section 8.4.2]{an2018chekanov}. The Reeb chord $c$ goes from $\varLambda_2$ to $\varLambda_1$ if $e$ is oriented as $v_1 \to v_2$ in $Q$, and comes from the fact that the boundary of $\varLambda_1 \cup \varLambda_2$ in the blue copy of $B^{2n-2}$ is a $(n-2)$-dimensional standard Hopf link (see \cref{fig:locally_at_frozen_edge_vertex}). The generators $d_{e,v_i}$ are those of the corresponding $(n-3)$-dimensional unknot in the boundary of the $B^{2n-4}$. The generators $a_{e,v_1}',a_{e,v_2}',h_e^\ast,h_e$ are the generators of the standard $(n-2)$-dimensional standard Hopf link in the boundary of the blue copy of $B^{2n-2}$.

			The Morse flow trees counted in $\partial a_{e,v_1}$ and $\partial a_{e,v_2}$ are depicted in \cref{fig:flow_tree8,fig:flow_tree9,fig:flow_tree10,fig:flow_tree11,fig:flow_tree12,fig:flow_tree13} and those counted in $\partial c$ are depicted in \cref{fig:flow_tree14}. The Morse flow trees counted in $\partial a_{e,v_1}'$ and $\partial a_{e,v_2}'$ are similar to those depicted in \cref{fig:flow_tree4,fig:flow_tree5,fig:flow_tree6,fig:flow_tree7} except that they appear in the boundary of the blue copy of $B^{2n-2}$.
		\end{description}

		\begin{figure}[!htb]
			\centering
			\includegraphics{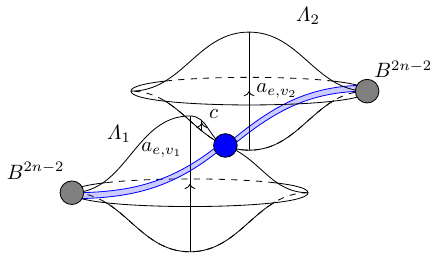}
			\caption{The more accurate front projection of the Legendrian $\varLambda_1 \cup \varLambda_2$ with boundary used to compute $\mathcal A_e$ for $e \in F_1$.}\label{fig:frozen_edge_local_chord}
		\end{figure}
		\begin{figure}[!htb]
			\centering
			\begin{subfigure}{0.32\textwidth}
					\centering
					\includegraphics{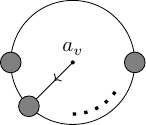}
					\caption{}\label{fig:flow_tree1}
			\end{subfigure}
			\begin{subfigure}{0.32\textwidth}
					\centering
					\includegraphics{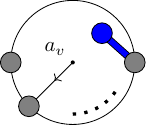}
					\caption{}\label{fig:flow_tree2}
			\end{subfigure}
			\begin{subfigure}{0.32\textwidth}
					\centering
					\includegraphics{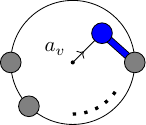}
					\caption{}\label{fig:flow_tree3}
			\end{subfigure}
			\begin{subfigure}{0.32\textwidth}
					\centering
					\includegraphics{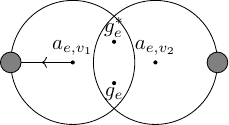}
					\caption{}\label{fig:flow_tree4}
			\end{subfigure}
			\begin{subfigure}{0.32\textwidth}
					\centering
					\includegraphics{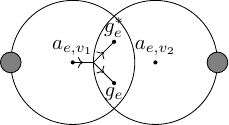}
					\caption{}\label{fig:flow_tree5}
			\end{subfigure}
			\begin{subfigure}{0.32\textwidth}
					\centering
					\includegraphics{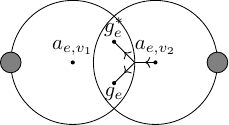}
					\caption{}\label{fig:flow_tree6}
			\end{subfigure}
			\begin{subfigure}{0.32\textwidth}
					\centering
					\includegraphics{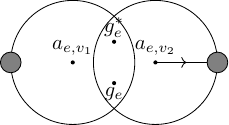}
					\caption{}\label{fig:flow_tree7}
			\end{subfigure}
			\begin{subfigure}{0.32\textwidth}
					\centering
					\includegraphics{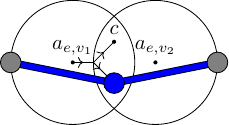}
					\caption{}\label{fig:flow_tree8}
			\end{subfigure}
			\begin{subfigure}{0.32\textwidth}
					\centering
					\includegraphics{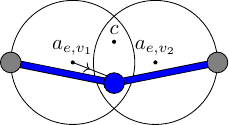}
					\caption{}\label{fig:flow_tree9}
			\end{subfigure}
			\begin{subfigure}{0.32\textwidth}
					\centering
					\includegraphics{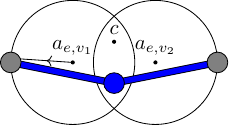}
					\caption{}\label{fig:flow_tree10}
			\end{subfigure}
			\begin{subfigure}{0.32\textwidth}
					\centering
					\includegraphics{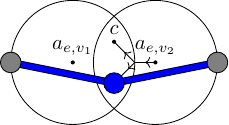}
					\caption{}\label{fig:flow_tree11}
			\end{subfigure}
			\begin{subfigure}{0.32\textwidth}
					\centering
					\includegraphics{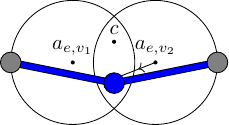}
					\caption{}\label{fig:flow_tree12}
			\end{subfigure}
			\begin{subfigure}{0.32\textwidth}
					\centering
					\includegraphics{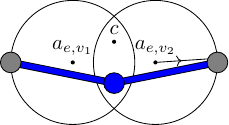}
					\caption{}\label{fig:flow_tree13}
			\end{subfigure}
			\begin{subfigure}{0.32\textwidth}
					\centering
					\includegraphics{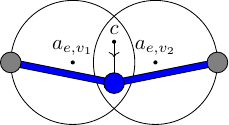}
					\caption{}\label{fig:flow_tree14}
			\end{subfigure}
			\caption{Morse flow trees describing the differential of $\mathcal A_x$.}
		\end{figure}
		\begin{rmk}
			We make a few comments regarding gradings and signs in the above computation of $\mathcal A_x$.
			\begin{enumerate}
				\item We use the cohomological grading convention in $\mathcal A_x$. The degree of a Reeb chord is given by the negative of its Conley--Zehnder index plus one, and the differential increases degree by one. This grading is the negative of the usual homological grading used in e.g.\@ \cite{chekanov2002differential,ekholm2005contact,ekholm2015legendrian}. The main reason for this choice is to ensure that the statement \cref{cor:orlov} makes sense, since the so-called surgery formula \cite[Theorem 5.8; Theorem 83; Theorem 1.1]{bourgeois2012effect,ekholm2017duality,asplund2020chekanov} holds with this choice of grading.
				\item The Maslov class of every Legendrian involved vanishes, and to define the degree of the generators corresponding to Reeb chords between two different components we have made an implicit choice of basepoint on each component. We have further chosen the same Maslov potential on every component.
				\item The signs appearing in $\mathcal A_x$ a priori depend on a number of choices, all of which we have left implicit. We however note that there is a unique spin structure, decreasing the number of choices by one.

				Let us compare the signs appearing in $\partial a_{e,v_1}$ and $\partial a_{e,v_2}$ for a non-frozen edge $e \colon v_1 \to v_2$. Having fixed choices so that $\partial a_{e,v_2} = b_{e,v_1} + g_eg_e^\ast$, the sign in front of $g_e^\ast g_e$ (which is denoted by $\nu_{\text{end}}$ in \cite[Theorem 1.1]{karlsson2017compute}) only depends on the order of the generators and differs from the sign in front of $g_eg_e^{\ast}$ in $\partial a_{e,v_2}$, respectively. The same reasoning applies to $\partial a'_{e,v_1}$ and $\partial a'_{e,v_2}$ for $e \in F_1$. This is also the reason for the sign difference between the terms $h_e c$ and $c h_e$ present in $\partial a_{e,v_1}$ and $\partial a_{e,v_2}$. The sign on front of $a'_{e,v_i}$ in $\partial a_{e,v_i}$ is determined by demanding $\partial^2 = 0$, using the assumption that we have equipped the moduli spaces of holomorphic curves with coherent orientations. This is also how the signs in $\partial a_v$ for $v\in F_0$ and $\partial c$ are determined.

				Since we have made no explicit choices and only computed relative sign differences, the actual Chekanov--Eliashberg dg-algebra may differ from $\mathcal A_x$ by a composition of automorphisms $z \mapsto -z$ for some generator $z \in \mathcal A_x$. The isomorphism class is however unchanged, and the relative signs suffices for our purposes.
			\end{enumerate}
		\end{rmk}
		\begin{proof}[Proof of \cref{thm:main}]
			By \cref{lma:colimit_loop} it suffices to prove that $\mathcal A$ is quasi-isomorphic to the relative Ginzburg algebra $\mathscr G^\ast_n(Q,F)$ (see \cref{dfn:rel_ginzburg}). Consider the map $\tau \colon \mathcal A \to \mathscr G^\ast_n(Q,F)$ defined on generators as
			\begin{align*}
				\tau(b_{e,v}) &= \begin{cases}
					g_e g_e^\ast, & \text{if $e \colon v \to \bullet$} \\
					-g_e^\ast g_e, & \text{if $e \colon \bullet \to v$}
				\end{cases},\; \text{$e$ non-frozen}, \quad \tau(b_v) = [g_e,g_e^\ast]_v^{Q_1 \smallsetminus F_1},\; \text{$v$ frozen} \\
				\tau(a_v) &= a_v,\; \text{$v$ non-frozen}, \quad \tau(g_e) = g_e, \quad \tau(g_e^\ast) = g_e^\ast, \quad \tau(h_e) = h_e \\
				\tau(x) &= 0, \; \text{all other generators},
			\end{align*}
			and extended to the whole of $\mathcal A$ by linearity and multiplicativity. Furthermore define $\varphi \colon \mathscr G^\ast_n(Q,F) \to \mathcal A$ on generators by
			\begin{align*}
				\varphi(a_v) = a_v - \sum_{e \supset v} a_{e,v}, \; \text{$v$ non-frozen}, \quad \varphi(x) = x, \; \text{all other generators},
			\end{align*}
			and extended to the whole of $\mathscr G^\ast_n(Q,F)$ by linearity and multiplicativity. Finally define $K \colon \mathcal A \to \mathcal \mathcal A$ on generators by
			\begin{align*}
				K(b_{e,v}) &= a_{e,v}, \quad K(d_{e,v}) = \begin{cases}
					h_ec + a'_{e,v}, & e \colon v \to \bullet \\
					-ch_e + a'_{e,v}, & e \colon \bullet \to v
				\end{cases}, \quad K(b_v) = -a_v + \sum_{e \supset v} a_{e,v} \\
				K(h_e^\ast) &= c, \quad K(x) = 0, \; \text{all other generators},
			\end{align*}
			and extended to the whole of $\mathcal A$ by linearity and as an $(\id,\varphi\tau)$-derivation. By an easy but tedious check we have that $\tau$ and $\varphi$ are chain maps and $K$ is a chain homotopy $\varphi \tau \simeq \id$. Applying the gluing formula in \cref{lma:colimit_loop} finishes the first part of the proof.

			By \cite[Remark 2.8]{asplund2020chekanov} there is a canonical dg-subalgebra $\mathcal B \subset CE^\ast(\varLambda(Q,F);Z_0(Q))$ that is quasi-isomorphic to the dg-subalgebra generated by the boundary $\partial \overline \varLambda(Q,F) \subset \partial P_0(F)$. By the computation of the dg-algebra $\mathcal A$ we have that $\mathcal B$ is generated by $\left\{b_v,d_{e,v},a'_{e,v},h_e,h_e^\ast\right\}$. By construction the boundary $\partial \overline \varLambda(Q,F) \subset \partial P_0(F)$ is a smooth Legendrian unknot link so that when attaching a top Weinstein handle to each component of $\partial \overline \varLambda(Q,F)$ and call the result $P(F)$, we have that $P(F)$ is a disjoint collection of plumbings of copies of $T^\ast S^{n-1}$ according to the quiver $F$. Repeating the above argument with a trivial frozen subquiver with $Q = F$ (or appealing to \cite[Corollary 4.16]{asplund2021simplicial}) yields the result. To spell it out explicitly, we define maps $\sigma \colon \mathcal B \to \mathscr G^\ast_{n-1}(F)$, $\varepsilon \colon \mathscr G^\ast_{n-1}(F) \to \mathcal B$ and $J \colon \mathcal B \to \mathcal B$ on generators as
			\begin{align*}
				\sigma(d_{e,v}) &= \begin{cases}
					h_eh_e^\ast, & \text{if $e \colon v \to \bullet$} \\
					-h_e^\ast h_e, & \text{if $e \colon \bullet \to v$}
				\end{cases}, \quad \sigma(a'_{e,v}) = 0, \quad \sigma(x) = x , \; \text{all other generators}.\\
				\varepsilon(b_v) &= b_v - \sum_{e \supset v} a'_{e,v}, \quad \varepsilon(x), \; \text{all other generators} \\
				J(d_{e,v}) &= a'_{e,v}, \quad J(x) = 0 \; \text{all other generators}.
			\end{align*}
			We extend $\sigma$ and $\varepsilon$ to all of $\mathcal B$ and $\mathscr G^\ast_{n-1}(F)$ by linearity and multiplicativity, and extend $J$ to all of $\mathcal B$ by linearity and as an $(\id,\varepsilon \sigma)$-derivation. As above, one can check that $\sigma$ and $\varepsilon$ are chain maps and $J$ specifies a chain homotopy equivalence $\varepsilon \sigma \simeq \id$.

			Finally a straightforward check shows that the following diagram commutes
			\[
				\begin{tikzcd}[row sep=scriptsize, column sep=scriptsize]
					\mathcal B \rar[hook]{\text{incl.}} \dar{\sigma} & CE^\ast(\varLambda(Q,F);Z_0(Q)) \dar{\tau} \\
					\mathscr G^{\ast}_{n-1}(F) \rar{G} & \mathscr G^\ast_n(Q,F)
				\end{tikzcd},
			\]
			where $G \colon \mathscr G^\ast_{n-1}(F) \to \mathscr G^\ast_n(Q,F)$ is the map defined in \cref{lma:inclusion_rel_ginz}.
		\end{proof}

\bibliographystyle{alpha}
\bibliography{relginz}
\end{document}